\documentclass[reqno]{amsart}

\usepackage[letterpaper,margin=1.25in]{geometry}

\usepackage{amsmath,amsthm,amssymb}
\usepackage[all]{xy}

\usepackage{xcolor}
\definecolor{darkgreen}{rgb}{0,0.45,0}
\usepackage{ifpdf}
\ifpdf  \usepackage[pdftex,colorlinks,urlcolor=blue,citecolor=darkgreen,linkcolor=darkgreen,linktocpage]{hyperref}
\else   \usepackage[hypertex,colorlinks,urlcolor=blue,citecolor=darkgreen,linkcolor=darkgreen,linktocpage]{hyperref}
\fi

\makeatletter
\def\subsection{\@startsection{subsection}{2}%
  \z@{.5\linespacing\@plus.7\linespacing}{.5\linespacing}%
  {\normalfont\bfseries}}
\makeatother

\newtheorem{de}{Definition}[section]
\newtheorem{lem}[de]{Lemma}
\newtheorem{prop}[de]{Proposition}
\newtheorem{cor}[de]{Corollary}
\newtheorem{thm}[de]{Theorem}

\theoremstyle{remark}
\newtheorem{rem}[de]{Remark}
\newtheorem{ex}[de]{Example}

\newcommand{\dfn}[1]{\textbf{\boldmath{#1}}}

\newcommand{\ra}{\to}
\newcommand{\lra}       {\longrightarrow}
\newcommand{\llra}[1]   {\stackrel{#1}{\lra}}  % labelled long right arrow

\newcommand{\N}{\ensuremath{\mathbb{N}}}
\newcommand{\Z}{\ensuremath{\mathbb{Z}}}

\newcommand{\R}{\ensuremath{\mathbb{R}}}

\newcommand{\cSD}{\mathcal{SD}}
\newcommand{\cSV}{\mathcal{SV}}
\newcommand{\cDV}{\mathcal{DV}}
\newcommand{\cF}{\mathcal{FFV}}
\newcommand{\fine}{\mathcal{FV}}
\newcommand{\proj}{\mathcal{PV}}
\newcommand{\haus}{\mathcal{HT}}

\newcommand{\DVect}{\mathrm{DVect}}

\DeclareMathOperator{\im}{Im}
\newcommand{\supp}{\text{supp}}

\newcommand{\eps}{\epsilon}

\title{Diffeological Vector Spaces}

\author{J. Daniel Christensen}
\address{Department of Mathematics,
         University of Western Ontario,
         London, Ontario,
         Canada}
\email{jdc@uwo.ca}

\author{Enxin Wu}
\address{Department of Mathematics,
         Shantou University,
         Guangdong, P.R. China}
\email{exwu@stu.edu.cn}

\thanks{The second author was partially supported by NNSF of China (No.\ 112530) and 
STU Scientific Research Foundation for Talents (No.\ 760179).}

\date{}

\begin{document}

\begin{abstract}
We study the relationship between many natural conditions that one can put 
on a diffeological vector space: being fine or projective,
having enough smooth (or smooth linear) functionals to separate points,
having a diffeology determined by the smooth linear functionals,
having fine finite-dimensional subspaces,
and having a Hausdorff underlying topology.
Our main result is that the majority of the conditions fit into a total order.
We also give many examples in order to show which implications do not hold,
and use our results to study the homological algebra of diffeological vector spaces.
\end{abstract}

% 57-XX Manifolds and cell complexes
%   57Pxx Generalized manifolds [See also 18F15]
%   * 57P99 None of the above, but in this section
% 46-XX   "Functional analysis 
%   46Sxx   Other (nonclassical) types of functional analysis
%   * 46S99   "None of the above, but in this section"


\maketitle

\tableofcontents

\section{Introduction}

Diffeological spaces are elegant generalizations of manifolds that include
a variety of singular spaces and infinite-dimensional spaces.
Many vector spaces that arise in applications are naturally equipped with
a compatible structure of a diffeological space.
Examples include $C^{\infty}(M, \R^n)$ for a manifold (or even a diffeological space) $M$,
spaces of smooth or holomorphic sections of vector bundles,
tangent spaces of diffeological spaces (as defined in~\cite{CW}),
smooth linear duals of all of these spaces, etc.
Such objects are called diffeological vector spaces and are the topic of this paper.

Diffeological vector spaces have been studied by Iglesias-Zemmour in~\cite{I1,I2}.
He used them to define diffeological manifolds, and developed the theory of
\emph{fine} diffeological vector spaces, a particularly
well-behaved kind that forms the beginning of our story.
In~\cite{KM}, Kriegl and Michor studied topological vector spaces equipped with
a smooth structure, and their examples can be regarded as diffeological vector spaces.
Diffeological vector spaces were used in the study of tangent spaces of diffeological
spaces in~\cite{V} and~\cite{CW}.
Wu investigated the homological algebra of \emph{all} diffeological vector spaces~\cite{Wu},
and the present paper builds heavily on this foundation.

In this paper, we study some natural conditions that one can put on a 
diffeological vector space, and show that the majority of them fit into
a total order.
In order to state our results, we briefly introduce the conditions here,
making use of some background material summarized in Section~\ref{se:background}.

Any vector space has a smallest diffeology making it into a diffeological
vector space.  This is called the \dfn{fine} diffeology, and we write $\fine$
for the collection of vector spaces with the fine diffeology.
We write $\cF$ for the collection of diffeological vector spaces whose
finite-dimensional subspaces (with the induced diffeology) are all fine.

A diffeological vector space $V$ is \dfn{projective} if for every linear
subduction $f : W_1 \to W_2$ and every smooth linear map $g : V \to W_2$,
there exists a smooth linear map $h : V \to W_1$ such that $g = f \circ h$.
We write $\proj$ for the collection of projective diffeological vector spaces.

A diffeological vector space $V$ is in $\cSD$ (resp. $\cSV$)
if the smooth (resp. smooth linear) functionals $V \to \R$ separate points of $V$.
That is, for each $x$ and $y$ in $V$ with $x \neq y$, such a functional $f$ can be found so that
$f(x) \neq f(y)$.

Each diffeological space has a natural topology called the $D$-topology.
We write $\haus$ for the collection of diffeological vector spaces whose
$D$-topologies are Hausdorff.

The last letter of the abbreviation is $\mathcal{V}$, $\mathcal{D}$ or $\mathcal{T}$
depending on whether the condition depends on the structure as a diffeological vector space,
a diffeological space, or a topological space.

We now state the main results of the paper.

\begin{thm}\label{thm:main1}
We have the following chain of containments:
\[
  \fine \subset \proj \subset \cSV \subseteq \cSD \subset \cF
  \quad\text{and}\quad
  \cSD \subset \haus ,
\]
where $\subset$ indicates proper containment.  
Neither of $\haus$ and $\cF$ contains the other.
\end{thm}

We do not know whether the containments $\cSV \subseteq \cSD$ and $\cSD \subseteq \cF \cap \haus$ are proper.

The property of being finite-dimensional does not imply, nor is it
implied by, any of the properties considered above.
However, under this assumption, most of the properties agree.

\begin{thm}\label{thm:main2}
When restricted to finite-dimensional vector spaces, the collections
$\fine, \proj, \cSV, \cSD$ and $\cF$ agree.
\end{thm}

Indeed, $\fine$ and $\cF$ clearly agree for finite-dimensional spaces,
so the containments must collapse to equalities.
Note that we prove part of Theorem~\ref{thm:main2} (see Theorem~\ref{thm:fd:fine<=>SV}) on the way to
proving Theorem~\ref{thm:main1}.

The final property we consider is the following.
Write $\cDV$ for the collection of diffeological vector spaces $V$ such
that a function $p : \R^n \to V$ is smooth if and only if $\ell \circ p : \R^n \to \R$
is smooth for each smooth linear functional $\ell : V \to \R$.
Except for the inclusion $\fine \subset \cDV$, the class $\cDV$ is independent
of all of the others we have considered.
However, under this assumption, we again find that many of the other conditions agree.

\begin{thm}\label{thm:main3}
When restricted to $V$ in $\cDV$, the collections $\cSV, \cSD, \cF$ and $\haus$ agree.
\end{thm}

The proofs of the containments, and the examples showing that many 
inclusions do not hold, are spread throughout Section~\ref{se:dvs}.
For example, we show $\fine \subset \proj$ in Example~\ref{ex:projective} and Proposition~\ref{prop:fine=>PV}, 
$\proj \subset \cSV$ in Proposition~\ref{prop:PV=>SV} and Remark~\ref{rem:SV>P}(1),
$\cSD \subset \haus$ in Proposition~\ref{prop:SD=>H} and Example~\ref{ex:H=/=>SD},
$\haus \nsubseteq \cF$ in Example~\ref{ex:H=/=>SD},
and both $\cSD \neq \cF$ and 
$\cF \nsubseteq \haus$ in Proposition~\ref{prop:FFV-not-haus}.
That $\proj \nsubseteq \cDV$ is Proposition~\ref{prop:d-vs-slf},
and the proof of Theorem~\ref{thm:main3} is in Proposition~\ref{prop:DV=>equiv}.
The longest argument, which is the proof that $\cSD \subseteq \cF$,
is deferred until Section~\ref{se:SD=>FDF}.
Along the way, we also prove other results, such as the fact that a
diffeological vector space $V$ is fine if and only if every linear functional on $V$ is smooth,
and some necessary conditions for diffeological vector spaces
and free diffeological vector spaces to be projective.
In Section~\ref{se:applications}, we give some applications of our results to
the homological algebra of diffeological vector spaces.
For example, we show that every finite-dimensional subspace of a diffeological
vector space in $\cSV$ is a smooth direct summand.

We thank Chengjie Yu for the argument used in Case 1 of the proof of
Theorem~\ref{thm:SD=>FDF} in Section~\ref{se:SD=>FDF}
and the referee for many comments that helped improve the exposition.

\section{Background and conventions}\label{se:background}

In this section, we briefly recall some background on diffeological spaces.
For further details, we recommend the standard textbook~\cite{I2}.
For a concise introduction to diffeological spaces, we recommend~\cite{CSW},
particularly Section~2 and the introduction to Section~3.

\begin{de}[\cite{So2}]\label{de:diffeological-space}
A \dfn{diffeological space} is a set $X$
together with a specified set of functions $U \to X$ (called \dfn{plots})
for each open set $U$ in $\R^n$ and each $n \in \N$,
such that for all open subsets $U \subseteq \R^n$ and $V \subseteq \R^m$:
\begin{enumerate}
\item (Covering) Every constant function $U \to X$ is a plot.
\item (Smooth Compatibility) If $U \to X$ is a plot and $V \to U$ is smooth,
then the composite $V \to U \to X$ is also a plot.
\item (Sheaf Condition) If $U=\cup_i U_i$ is an open cover
and $U \to X$ is a function such that each restriction $U_i \to X$ is a plot,
then $U \to X$ is a plot.
\end{enumerate}

A function $f:X \rightarrow Y$ between diffeological spaces is
\dfn{smooth} if for every plot $p:U \to X$ of $X$,
the composite $f \circ p$ is a plot of $Y$.
\end{de}

The category of diffeological spaces and smooth maps is complete and cocomplete.
Given two diffeological spaces $X$ and $Y$,
we write $C^\infty(X,Y)$ for the set of all smooth maps from $X$ to $Y$.
An isomorphism in the category of diffeological spaces will be called a \dfn{diffeomorphism}.

Every manifold $M$ is canonically a diffeological space with the
plots taken to be all smooth maps $U \to M$ in the usual sense.
We call this the \dfn{standard diffeology} on $M$.
It is easy to see that smooth maps in the usual sense between
manifolds coincide with smooth maps between them with the standard diffeology. 

For a diffeological space $X$ with an equivalence relation~$\sim$,
the \dfn{quotient diffeology} on $X/{\sim}$ consists of all functions
$U \to X/{\sim}$ that locally factor through the quotient map $X \to X/{\sim}$ via plots of $X$.
A \dfn{subduction} is a map diffeomorphic to a quotient map.
That is, it is a map $X \to Y$ such that the plots in $Y$
are the functions that locally lift to $X$ as plots in $X$.

For a diffeological space $Y$ and a subset $A$ of $Y$,
the \dfn{sub-diffeology} consists of all functions $U \to A$ such that 
$U \to A \hookrightarrow Y$ is a plot of $Y$.
An \dfn{induction} is an injective smooth map $A \to Y$ such that a
function $U \to A$ is a plot of $A$ if and only if $U \to A \to Y$ is a plot of $Y$.

For diffeological spaces $X$ and $Y$, the \dfn{product diffeology} on $X \times Y$
consists of all functions $U \to X \times Y$ whose components $U \to X$
and $U \to Y$ are plots of $X$ and $Y$, respectively.

The \dfn{discrete diffeology} on a set is the diffeology whose plots are
the locally constant functions.
The \dfn{indiscrete diffeology} on a set is the diffeology in which
every function is a plot.

We can associate to every diffeological space the following topology:

\begin{de}[\cite{I1}]
Let $X$ be a diffeological space.
A subset $A$ of $X$ is \dfn{$D$-open} if $p^{-1}(A)$ is open in $U$ for
each plot $p : U \to X$.
The collection of $D$-open subsets of $X$ forms a topology on $X$ called the \dfn{$D$-topology}.
\end{de}

\begin{de}
 A \dfn{diffeological vector space} is a vector space $V$ with
 a diffeology such that addition $V \times V \to V$ and scalar multiplication
 $\R \times V \to V$ are smooth.
\end{de}

Let $V$ be a diffeological vector space. We write $L^{\infty}(V, \R)$ for the set of all smooth 
linear maps $V \to \R$, and $L(V, \R)$ for the set of all linear maps $V \to \R$.
We write $\DVect$ for the category of diffeological vector spaces and smooth linear maps.

\medskip
\noindent
\textbf{Conventions}
\smallskip

Throughout this paper, we use the following conventions.
Every subset of a diffeological space is equipped with the sub-diffeology
and every product is equipped with the product diffeology.
Every vector space is over the field $\R$ of real numbers, and every linear map is $\R$-linear.
By a subspace of a diffeological vector space, we mean a linear subspace with the sub-diffeology.
All manifolds are smooth, finite-dimensional, Hausdorff, second countable and 
without boundary, and are equipped with the standard diffeology.

\section{Diffeological vector spaces}\label{se:dvs}

In this section, we study a variety of conditions that a diffeological vector space
can satisfy.
Together, the results described here give the theorems stated in the introduction.
In addition, we present some auxiliary results, and give many examples and counterexamples.

\subsection{Fine diffeological vector spaces}

 In this subsection, we recall background on the fine diffeology, and then
 give two new characterizations.

 Given a vector space $V$, the set of all diffeologies on $V$ each of which makes $V$ 
 into a diffeological vector space, ordered by inclusion, is a complete lattice.
 This follows from~\cite[Proposition~4.6]{CW}, taking $X$ to be a point.
 The largest element in this lattice is the indiscrete diffeology,
 which is usually not interesting. Another extreme has the following special name 
 in the literature:

\begin{de}
 The \dfn{fine} diffeology on a vector space $V$ is the smallest diffeology on $V$
 making it into a diffeological vector space.
\end{de}

 For example, the fine diffeology on $\R^n$ is the standard diffeology.

\begin{rem}\label{rem:fineness}
 The fine diffeology is generated by the injective linear maps $\R^n \to V$ (\cite[3.8]{I2}).
 That is, the plots of the fine diffeology are the functions $p : U \to V$ such that 
 for each $u \in U$, there are an open neighbourhood $W$ of $u$ in $U$,
 an injective linear map $i : \R^n \to V$ for some $n \in \N$, and a smooth map $f : W \to \R^n$
 such that $p|_W = i \circ f$.

 One can show that if $V$ is any diffeological vector space and $p: W \to V$ is a plot
 that factors smoothly through some linear injection $\R^n \to V$,
 then every factorization of $p$ through a linear injection $\R^m \to V$ is smooth.
 It follows that every subspace of a fine diffeological vector space is fine (\cite{Wu}).
\end{rem}

 In fact, fineness of a diffeological vector space can be tested by smooth curves:

\begin{prop}\label{prop:fineness-on-curves}
 A diffeological vector space $V$ is fine if and only if for every plot $p:\R \ra V$ and every 
 $x \in \R$, there exist an open neighbourhood $W$ of $x$ in $\R$, an injective linear map $i:\R^n \ra V$ for 
 some $n \in \N$, and a smooth map $f:W \ra \R^n$ such that $p|_W = i \circ f$.
\end{prop}

\begin{proof}
 ($\Rightarrow$) This follows from the description of the fine diffeology in Remark~\ref{rem:fineness}.

 ($\Leftarrow$) Under the given assumptions, we will prove that $V$ is fine.
 Let $q : U \ra V$ be a plot and let $u$ be a point in $U$.
 We first show that there is an open neighbourhood $W$ of $u$ in $U$ such that
 $q|_W$ lands in a finite-dimensional subspace of $V$.
 If not, then there exists a sequence $u_i$ in $U$ converging to $u$
 such that $\{ q(u_i) \mid i \in \Z^+ \}$ is linearly independent in $V$. We may assume 
 that the sequence $u_i$ converges fast to $u$ (see~\cite[I.2.8]{KM}). By the Special Curve Lemma~\cite[I.2.8]{KM},
 there exists a smooth map $f:\R \ra U$ such that $f(1/i)=u_i$ and $f(0)=u$. Then 
 $q \circ f:\R \ra V$ is a plot which does not satisfy the hypothesis at $x=0$.

 So let $W$ be an open neighbourhood of $u$ in $U$ such that $q|_W$ factors as
 $i \circ g$, where $i : \R^m \ra V$ is a linear injection and $g : W \ra \R^m$ is a function.
 We will prove that $g$ is smooth.
 By Boman's theorem (see, e.g.,~\cite[Corollary~3.14]{KM}), it is enough to show that
 $g \circ r$ is smooth for every smooth curve $r : \R \ra W$.
 Since $i \circ g \circ r$ is smooth, our assumption implies that it locally factors smoothly
 through an injective linear map $\R^n \to V$.
 Then the last part of Remark~\ref{rem:fineness} implies that $g \circ r$ is locally smooth,
 and therefore smooth, as required.
\end{proof}

\begin{prop}\label{prop:fine<=>Linf=L}
 A diffeological vector space $V$ is fine if and only if $L^{\infty}(V, \R) = L(V, \R)$,
 i.e., if and only if every linear functional is smooth.
\end{prop}

\begin{proof}
 This follows from the proof of~\cite[Proposition~5.7]{Wu}. We give a direct proof here. 
 
 It is easy to check that if $V$ is fine, then every linear functional is smooth.

 To prove the converse, suppose that every linear functional $V \to \R$ is smooth.
 Let $p : U \to V$ be a plot and let $u \in U$.
 First we must show that when restricted to a neighbourhood of $u$,
 $p$ lands in a finite-dimensional subspace of $V$.
 If not, then there is a sequence $\{ u_j \}$ converging to $u$ such that the
 vectors $p(u_j)$ are linearly independent.
 Thus there is a linear functional $l : V \to \R$ such that $p(u_j)$ is sent to $1$ when $j$ is odd
 and $0$ when $j$ is even.  By assumption, $l$ is smooth.
 But $l \circ p$ is not continuous, contradicting the fact that $p$ is a plot.

 So now we know that $p$ locally factors through an injective linear map $i : \R^n \to V$.
 (Of course, $n$ may depend on the neighbourhood.)
 For each $1 \leq j \leq n$, there is a linear map $l_j : V \to \R$ such that $l_j \circ i$
 is projection onto the $j^{th}$ coordinate.  Since $l_j \circ p$ is smooth, it follows that
 the local factorizations through $\R^n$ are smooth.  Thus $V$ is fine.
\end{proof}

\subsection{Projective diffeological vector spaces}

\begin{de}
 A diffeological vector space $V$ is \dfn{projective} if for every linear
 subduction $f : W_1 \to W_2$ and every smooth linear map $g : V \to W_2$,
 there exists a smooth linear map $h : V \to W_1$ making the diagram
 \[
   \xymatrix{
                               &  W_1 \ar[d]^-f \\
     V \ar@{-->}[ur]^-h \ar[r]_-g &  W_2
   }
 \]
 commute.
 We write $\proj$ for the collection of projective diffeological vector spaces.
\end{de}

 We now describe what will be a recurring example in this paper.

\begin{de}
 The \dfn{free diffeological vector space generated by a diffeological space $X$}
 is the vector space $F(X)$ with basis consisting of the elements of $X$ and
 with the smallest diffeology
 making it into a diffeological vector space and such that the natural map
 $X \to F(X)$ is smooth.
\end{de}

This has the universal property that for any diffeological vector space $V$,
every smooth map $X \to V$ extends uniquely to a smooth linear map $F(X) \to V$.
Also, every plot in $F(X)$ is locally of the form $u \mapsto \sum_{i=1}^k r_i(u) [p_i(u)]$
for smooth functions $r_i : U \to \R$ and $p_i : U \to X$,
where for $x \in X$, $[x]$ denotes the corresponding basis vector in $F(X)$.
See~\cite[Proposition~3.5]{Wu} for details.

\begin{ex}\label{ex:projective}
 By~\cite[Corollary~6.4]{Wu}, when $M$ is a manifold, $F(M)$ is projective.
 However, by~\cite[Theorem~5.3]{Wu}, $F(X)$ is fine if and only if $X$ is discrete.
 So not every projective diffeological vector space is fine.
\end{ex}

\begin{prop}[{\cite[Corollary~6.3]{Wu}}]\label{prop:fine=>PV}
 Every fine diffeological vector space is projective.
\end{prop}

\begin{proof}
 This follows immediately from Proposition~\ref{prop:fine<=>Linf=L}.
 One can take $h$ to be $k \circ g$, where $k$ is a linear section of $f$
 (which is not necessarily smooth).
\end{proof}

 Projective diffeological vector spaces and the homological algebra of
 diffeological vector spaces are studied further in~\cite{Wu}.

\subsection{Separation of points}

\begin{de}
 Let $X$ be a diffeological space.
 A set $A$ of functions with domain $X$ is said to \dfn{separate points} if
 for any $x, y \in X$ with $x \neq y$, there exists $f \in A$ such that $f(x) \neq f(y)$.
 We say that the \dfn{smooth functionals separate points} if $C^{\infty}(X, \R)$ separates points.
 We write $\cSD'$ for the collection of all such diffeological spaces $X$
 and $\cSD$ for the diffeological vector spaces whose underlying diffeological spaces are in $\cSD'$.
 If $V$ is a diffeological vector space, we say that the \dfn{smooth linear functionals separate points}
 if $L^{\infty}(V, \R)$ separates points,
 and we write $\cSV$ for the collection of all such diffeological vector spaces $V$.
\end{de}

 We will establish basic properties of such diffeological vector spaces below,
 and show that many familiar diffeological vector spaces have this property.
 Clearly, $\cSV \subseteq \cSD$.

\begin{ex}\label{ex:SD,SV}
 Every fine diffeological vector space is in $\cSV$, since the coordinate functions
 with respect to any basis are smooth and linear.
 Every manifold is in $\cSD'$, since the products of local coordinates with bump
 functions separate points
 (or by Whitney's embedding theorem).
\end{ex}

\pagebreak[2]  % Prefer a pagebreak here to within the list:
\begin{prop}\label{prop:S-sub-(co)product}\
\begin{enumerate}
 \item If $W \ra V$ is a smooth linear injective map between diffeological vector spaces
 and $V \in \cSV$, then $W \in \cSV$. In particular, $\cSV$ is closed under taking subspaces.
 
 \item Let $\{V_i\}_{i \in I}$ be a set of diffeological vector spaces. 
 Then $\prod_{i \in I} V_i \in \cSV$ if and only if each $V_i \in \cSV$.
 
 \item Let $\{V_i\}_{i \in I}$ be a set of diffeological vector spaces. 
 Then $\oplus_{i \in I} V_i \in \cSV$ if and only if each $V_i \in \cSV$, 
 where $\oplus_{i \in I} V_i$ is the coproduct in $\DVect$ (see~\cite[Proposition~3.2]{Wu}).
\end{enumerate}
\end{prop}

\begin{proof}
 This is straightforward.
\end{proof}

\begin{prop}\label{prop:S-function}
 If $V \in \cSV$ and $X$ is a diffeological space, then $C^\infty(X,V) \in \cSV$.
\end{prop}

\begin{proof}
  This follows from the fact that every evaluation map $C^\infty(X,V) \ra V$ is smooth and linear.
\end{proof}

\begin{samepage}
\begin{prop}\label{prop:S-free}
 The following are equivalent:
  \begin{enumerate}
    \item $X \in \cSD'$.

    \item $F(X) \in \cSV$.

    \item $F(X) \in \cSD$.
  \end{enumerate}
\end{prop}
\end{samepage}

\begin{proof}
 It is enough to prove $(1) \Rightarrow (2)$, 
 since $(2) \Rightarrow (3) \Rightarrow (1)$ are straightforward.
 Let $v \in F(X)$ be non-zero.
 It suffices to show that there is a smooth linear functional $F(X) \to \R$ which
 is non-zero on $v$.
 Write $v = \sum_{i=1}^k r_i [x_i]$ with $k \geq 1$, $r_i$ nonzero for each $i$,
 and the $x_i$ distinct.
 Since $C^\infty(X,\R)$ separates points of $X$,
 there exists $f \in C^\infty(X,\R)$ such that $f(x_1)=1$ and $f(x_i)=0$ for each $i > 1$. 
 By the universal property of $F(X)$, $f$ extends to a smooth linear map $F(X) \to \R$
 which sends $v$ to $r_1$, which is nonzero.
\end{proof}

\begin{prop}\label{prop:PV=>SV}
 Every projective diffeological vector space is in $\cSV$.
\end{prop}

\begin{proof}
 By Example~\ref{ex:SD,SV}, every open subset $U$ of a Euclidean space is in $\cSD'$.
 So Proposition~\ref{prop:S-free} implies that $F(U)$ is in $\cSV$.
 Corollary~6.15 of~\cite{Wu} says that
 every projective diffeological vector space is a retract of a coproduct of $F(U)$'s in $\DVect$.
 Therefore, it follows from Proposition~\ref{prop:S-sub-(co)product}(3) and~(1) that
 every projective diffeological vector space is in $\cSV$.
\end{proof}

\begin{rem}\label{rem:SV>P}\
 \begin{enumerate}
  \item Not every diffeological vector space in $\cSV$ is projective. For example, let $V:=\prod_\omega \R$
   be the product of countably many copies of $\R$. By Proposition~\ref{prop:S-sub-(co)product}(2),
   $V$ is in $\cSV$. But~\cite[Example~4.3]{Wu} shows
   that $V$ is not projective.
   
  \item $\cSV$ is not closed under taking quotients in $\DVect$. For example, $F(\pi):F(\R) \ra F(T_\alpha)$ 
   is a linear subduction, where $\alpha$ is an irrational and
   $\pi:\R \ra T_\alpha := \R/(\Z+\alpha \Z)$ is the projection to the $1$-dimensional irrational torus.
   By Proposition~\ref{prop:S-free}, $F(\R)$ is in $\cSV$,
   but $F(T_\alpha)$ is not in $\cSV$ since $T_\alpha$ is not in $\cSD'$.
   In particular, the free diffeological vector space $F(T_{\alpha})$ is not projective,
   as observed in \cite[Example~4.3]{Wu}.
  \end{enumerate}
\end{rem}

 Here is an easy fact:

\begin{prop}\label{prop:SD=>H}
 The $D$-topology of every diffeological space in $\cSD'$ is Hausdorff.
 In particular, $\cSD \subseteq \haus$.
\end{prop}

\begin{proof}
 This follows from the fact that every smooth map is continuous when both domain and codomain are
 equipped with the $D$-topology.
\end{proof}

\begin{cor}
 If $F(X)$ is projective, then $X$ is Hausdorff.
\end{cor}

\begin{proof}
 If $F(X)$ is projective, then it is in $\cSD$ by Proposition~\ref{prop:PV=>SV},
 and so $X$ is in $\cSD'$ by Proposition~\ref{prop:S-free}.
 Thus $X$ is Hausdorff, by Proposition~\ref{prop:SD=>H}.
\end{proof}

 This gives another proof that free diffeological vector spaces are not always projective.
 For example, if a set $X$ with more than one point is equipped with the indiscrete diffeology,
 then the $D$-topology on $X$ is indiscrete as well, and hence $F(X)$ is not projective.

\begin{ex}\label{ex:H=/=>SD}
 The converse of Proposition~\ref{prop:SD=>H} does not hold.
 Write $C(\R)$ for the vector space $\R$ equipped with the continuous diffeology,
 so that a function $p : U \to C(\R)$ is a plot if and only if it is continuous~(\cite[Section~3]{CSW}).
 Then $C(\R)$ is a Hausdorff diffeological vector space, as the $D$-topology on $C(\R)$ is the usual topology.
 But one can show that $C^{\infty}(C(\R), \R)$ consists of constant functions~(\cite[Example~3.15]{CW}), so $C(\R)$ is not in $\cSD$.
\end{ex}

 We will use the following result in the next subsection.

\begin{thm}\label{thm:fd:fine<=>SV}
 Let $V$ be a finite-dimensional diffeological vector space.
 Then the following are equivalent:
 \begin{enumerate}
  \item $V$ is fine.
  \item $V$ is projective.
  \item $V$ is in $\cSV$.
 \end{enumerate}
\end{thm}

\begin{proof}
 By Propositions~\ref{prop:fine=>PV} and~\ref{prop:PV=>SV}, $(1) \implies (2) \implies (3)$, for all $V$.
 So it remains to prove $(3) \implies (1)$.
 Assume that $V$ is finite-dimensional and in $\cSV$.
 Choose a basis $f_1, \ldots, f_k$ for $L^{\infty}(V, \R)$, and use it to give a smooth linear map $f : V \to \R^k$.
 Note that $k \leq \dim V$.
 Since $V$ is in $\cSV$, $f$ is injective, and hence surjective.
 The diffeology on $\R^k$ is the fine diffeology, which is the smallest diffeology making
 it into a diffeological vector space.
 The map $f:V \to \R^k$ is a smooth linear bijection, so the diffeology on $V$ must be fine as well
 (and $f$ must be a diffeomorphism).
\end{proof}

The implication $(3) \implies (1)$ also follows from Proposition~\ref{prop:fine<=>Linf=L}, since
$V$ in $\cSV$ implies that $\dim L(V,\R) \geq \dim L^{\infty}(V, \R) \geq \dim V = \dim L(V, \R)$,
and so $L^{\infty}(V, \R) = L(V, \R)$.

\subsection{Diffeological vector spaces whose finite-dimensional subspaces are fine}

 Write $\cF$ for the collection of diffeological vector spaces whose
 finite-dimensional subspaces are fine.
 One motivation for studying this collection is the following.
 In~\cite{CW}, we defined a diffeology on Hector's tangent spaces~\cite{He} which makes
 them into diffeological vector spaces.  While they are not fine in general, we know of no
 examples that are not in $\cF$.

 As an example, one can show that $\prod_{\omega} \R$ is in $\cF$.
 This also follows from the next result, which is based on a suggestion of Y.~Karshon.

\begin{thm}\label{thm:slfsp=>fdf}
 Every diffeological vector space in $\cSV$ is in $\cF$.
\end{thm}

 This result is a special case of Theorem~\ref{thm:SD=>FDF} below, but we 
 provide a direct proof, since it follows easily from earlier results.

\begin{proof}
 Let $W$ be a finite-dimensional subspace of $V$ with $V \in \cSV$. 
 By Proposition~\ref{prop:S-sub-(co)product}(1), $W$ is in $\cSV$, 
 and then by Theorem~\ref{thm:fd:fine<=>SV}, $W$ is fine.
\end{proof}

\begin{rem}\
 \begin{enumerate}
  \item \label{rem:countable-dimension} Note that it is not in general true that every
  diffeological vector space in $\cSV$ is fine.
  For example, $\prod_\omega \R$ is in $\cSV$ by Remark~\ref{rem:SV>P}(1), but it is not fine.
  In fact,~\cite[Example~5.4]{Wu} showed that there is a countable-dimensional subspace of $\prod_{\omega} \R$ which
  is not fine.
  Incidentally, it follows that $\prod_{\omega} \R$ is not the colimit in $\DVect$ of its finite-dimensional subspaces,
  since fine diffeological vector spaces are closed under colimits (\cite[Property~6 after Definition~5.2]{Wu}).

  \item When $\R$ is equipped with the continuous diffeology (see Example~\ref{ex:H=/=>SD}),
        it is Hausdorff but is not in $\cF$.
        We will see in Proposition~\ref{prop:FFV-not-haus} that the reverse inclusion
        also fails to hold.
 \end{enumerate}
\end{rem}

The main result of this section is the following:

\begin{thm}\label{thm:SD=>FDF}
 Every diffeological vector space in $\cSD$ is in $\cF$.
\end{thm}

We defer the proof to Section~\ref{se:SD=>FDF}.

Furthermore, we have the following result:

\begin{prop}\label{prop:FFV-not-haus}
There exists a diffeological vector space which is in $\cF$ but which is not Hausdorff.
In particular, the containment of $\cSD$ in $\cF$ is proper.
\end{prop}

\begin{proof}
 Let $V$ be the vector space with basis $\R$,
 and for $r \in \R$ write $[r]$ for the corresponding basis vector of $V$.
 Let $f : \R \to \R$ be a bijection such that $f^{-1}(U)$ is dense in $\R$
 for every open neighbourhood $U$ of $0$ in $\R$.
 Define $p: \R \to V$ by $p(x) = [f(x)]$ and $\bar{p} : \R \to V$ by $\bar{p}(x) = [x]$.
 Equip $V$ with the vector space diffeology generated by $p$ and $\bar{p}$.
 In other words, $q: U \to V$ is a plot if and only if for every $u_0 \in U$
 there exist an open neighbourhood $U'$ of $u_0$ in $U$ and
 finitely many smooth functions $\alpha_i, \beta_i, \bar{\alpha}_j, \bar{\beta}_j: U' \to \R$
 such that for any $u \in U'$, 
 \begin{equation}\label{eq:local}
   q(u) = \sum_i \alpha_i(u) \, [f(\beta_i(u))] + \sum_j \bar{\alpha}_j(u) \, [\bar{\beta}_j(u)].
   \tag{$\dagger$}
 \end{equation}
 (In general, one should include terms with smooth multiples of arbitrary
 vectors in $V$, but since $\bar{\beta}_j$ can be constant,
 this case is included in the above.)
 First we show that the $D$-topology on $V$ is not Hausdorff.
 Suppose that $V_0$ and $V_1$ are disjoint $D$-open subsets of $V$
 containing $[0]$ and $[1]$, respectively.
 Then $U_0 := \bar{p}^{-1}(V_0)$ and $U_1 := \bar{p}^{-1}(V_1)$ are open in $\R$.
 Since $p$ and $\bar{p}$ are both bijections onto the subset of basis vectors in $V$,
 it follows that $p^{-1}(V_0) = f^{-1}(U_0)$ and $p^{-1}(V_1) = f^{-1}(U_1)$.
 Therefore, $p^{-1}(V_0)$ is dense in $\R$ and so $p^{-1}(V_1)$, which is contained in the
 complement, must not be open in $\R$, contradicting the assumption that $V_1$ is $D$-open.
 So $V$ is not Hausdorff.

 Next we show that $V$ is in $\cF$.
 It suffices to show that for any finite subset $A \subseteq \R$,
 the subspace $W$ spanned by $A$ has the fine diffeology.
 So let $q : U \to V$ be a plot which lands in $W$.
 We must show that for each $a \in A$, the component $q_a$ of $q$ is
 a smooth function $U \to \R$.
 This is a local property, so we choose $u_0 \in U$ and express $q$ in the form~\eqref{eq:local}.
 It suffices to handle each sum in~\eqref{eq:local} separately, so we begin by
 assuming that $q$ only has terms involving $f$.
 Let $A' = f^{-1}(A)$.
 By shrinking $U'$ if necessary, we can assume that:
 (1) for any $b' \in A'$, if $\beta_i(u_0) \neq b'$, then $\beta_i(u) \neq b'$ for all $u \in U'$;
 and (2) if $\beta_i(u_0) \neq \beta_j(u_0)$, then $\beta_i$ and $\beta_j$
 have disjoint images.
 Since $f$ is a bijection, we can rephrase these conditions as:
 (1') for any $b \in A$, if $f(\beta_i(u_0)) \neq b$, then $f(\beta_i(u)) \neq b$ for all $u \in U'$;
 and (2') if $f(\beta_i(u_0)) \neq f(\beta_j(u_0))$, then $f \circ \beta_i$ and $f \circ \beta_j$
 have disjoint images.
 Condition (1') implies that for $u \in U'$, $q_a(u)$ is the $a$-coefficient of
 \[
   \sum_{f(\beta_i(u_0)) = a} \alpha_i(u) \, [f(\beta_i(u))].
 \]
 Since $q(u)$ is in $W$, condition (2') implies that for $r \in \R \setminus A$, we must have
 \stepcounter{equation} % This gets rid of hyperref warning
 \begin{equation}\label{eq:sum}
   \sum_{f(\beta_i(u_0)) = a, f(\beta_i(u)) = r} \alpha_i(u) = 0.
   \tag{$\diamond$}
 \end{equation}
 And condition (1') implies that~\eqref{eq:sum} also holds for $r \in A \setminus \{a\}$,
 since the sum is empty in that case.
 Therefore, $q_a(u)$ can be expressed as
 \[
   \sum_{f(\beta_i(u_0)) = a} \alpha_i(u),
 \]
 which is a smooth function of $u \in U'$.

 The other sum in~\eqref{eq:local} is handled in a similar way, replacing $f$ by the
 identity function throughout.

 Finally, Proposition~\ref{prop:SD=>H} implies that $V$ is not in $\cSD$,
 giving the last claim.
\end{proof}

 Next we observe that if $V$ is projective (and hence in $\cSV$ and $\cSD$),
 it does not follow that all countable-dimensional subspaces of $V$ are fine.
 We will illustrate this with $V = F(\R)$.
 By~\cite[Corollary~6.4]{Wu}, $F(\R)$ is projective.

\begin{prop}\label{prop:cdpdvs}
 Let $A$ be a subset of $\R$, and let $V$ be the subspace of $F(\R)$ spanned by $A$. 
 Then $V$ is fine if and only if $A$ has no accumulation point in $\R$.
\end{prop}

 For example, $F(\R)$ is not fine.
 As a more interesting example, $A = \{ 1/n \mid n = 1, 2, \ldots \}$ spans a countable-dimensional
 subspace $V$ of $F(\R)$ which is not fine.
 It will follow from Proposition~\ref{prop:F(countable)} that $V$ is not free on any diffeological space.

\begin{proof}
 ($\Leftarrow$) Let $p: U \to V$ be a plot, where $U$ is open in some $\R^n$.
 Since $V$ is the span of $A$, there exist unique functions $h_a : U \to \R$
 such that $p(x) = \sum_{a \in A} h_a(x) [a]$.
 Since $A$ has no accumulation point in $\R$, for each $a$ in $A$ there
 exists a smooth bump function $\phi_a : \R \to \R$ which takes the value $1$ at $a$
 and $0$ at every other element of $A$.
 Associated to $\phi_a$ is a smooth linear map $\tilde{\phi}_a : F(\R) \to \R$ which
 sends $[a]$ to $1$ and all other basis elements from $A$ to $0$.
 Then $h_a = \tilde{\phi}_a \circ p$, which shows that $h_a$ is smooth.

 Next we show that locally $p$ factors through the span of a finite subset of $A$.
 Fix $u \in U$.
 As $V$ is a subspace of $F(\R)$, there is an open neighbourhood $U'$ of $u$ in $U$
 such that $p(x) = \sum_{j=1}^m f_j(x) [g_j(x)]$ for $x \in U'$,
 where $f_j$ and $g_j$ are smooth functions $U' \to \R$.
 Shrinking $U'$ if necessary, we can assume that it is contained in a compact subset of $U$.
 It follows that the image of each $g_j$ is contained in a compact subset of $\R$
 and therefore intersects only finitely many points of $A$.
 Since there are only finitely many $g_j$'s, $p|_{U'}$ factors through
 the span of $A'$ for some finite subset $A'$ of $A$.
 That is, $h_a(x) = 0$ for all $x \in U'$ and all $a \in A \setminus A'$.

 In summary, identifying the span of $A'$ with $\R^{A'}$, we have factored $p|_{U'}$
 as $U' \to \R^{A'} \to V$, where the first map is $x \mapsto (h_a(x))_{a \in A'}$ and the
 second map sends $f : A' \to \R$ to $\sum_{a \in A'} f(a)[a]$.

 ($\Rightarrow$) Now we prove that if $A$ has an accumulation point $a_0$ in $\R$, then $V$ is not fine.
 Pick a sequence $(a_i)$ in $A \setminus \{a_0\}$ that converges fast to $a_0$.
 Choose a smooth function $f : \R \to \R$ such that $f(x) \neq 0$ for $1/(2n+1) < x < 1/2n$ 
 for each $n \in \Z^+$, and $f(x) = 0$ for all other $x$.
 Choose another smooth function $g : \R \to \R$ such that $g(x) = a_n$ for $1/(2n+1) < x < 1/2n$
 for each $n \in \Z^+$, with no constraints on $g$ otherwise.
 It will necessarily be the case that $g(0) = a_0$,
 and such a smooth $g$ exists because the sequence was chosen to converge fast.
 Then the function $p : \R \to V$ defined by $p(x) = f(x)[g(x)]$ is smooth,
 but there is no open neighbourhood $U$ of $0$ so that $p|_U$ factors through
 a finite-dimensional subspace of $V$.
\end{proof}

 On the other hand, we have:

\begin{prop}\label{prop:F(countable)}
 Let $X$ be a diffeological space whose underlying set has cardinality less than
 the cardinality of $\R$.
 Then the following are equivalent:
 \begin{enumerate}
 \item $X$ is discrete.
 \item $F(X)$ is fine.
 \item $F(X)$ is projective.
 \item $F(X)$ is in $\cSV$.
 \item $F(X)$ is in $\cSD$.
 \item $F(X)$ is Hausdorff.
 \end{enumerate}
\end{prop}

\begin{proof}
 That (1) $\implies$ (2) is straightforward.
 The implications (2) $\implies$ (3) $\implies$ (4) and (5) $\implies$ (6) 
 follow from Propositions~\ref{prop:fine=>PV}, \ref{prop:PV=>SV} and~\ref{prop:SD=>H},
 while (4) $\implies$ (5) is clear.
 None of these use the assumption on the cardinality of $X$.

 It remains to prove that (6) $\implies$ (1).  
 Since the natural injective map $X \to F(X)$ is smooth, it is also continuous
 when $X$ and $F(X)$ are both equipped with the D-topology.
 Therefore, $X$ is Hausdorff.
 We must show that the diffeology on $X$ is discrete.
 Let $p: U \to X$ be a plot from a connected open subset $U$ of a Euclidean space.
 We will show that $p$ is constant.
 If not, then the image of $p$ contains two distinct points $x,x' \in X$ which 
 are connected by a continuous path $q : [0,1] \to X$.
 The image of $q$ is compact Hausdorff, and therefore normal.
 So by Urysohn's lemma, there is a continuous map $l : \im(q) \to \R$ which separates $x$ and $x'$.
 Hence, the image of the composite $l \circ q : [0,1] \to \im(q) \to \R$ has cardinality
 equal to the cardinality of $\R$,
 which is a contradiction, since $\im(q) \subseteq X$ has cardinality less than
 the cardinality of $\R$.
\end{proof}

Part of the above proof is based on the argument in~\cite{Ha}.
Note that we have proved that every Hausdorff diffeological space with cardinality less
than the cardinality of $\R$ is discrete.
The implication (2) $\implies$ (1) is also proved in~\cite[Theorem~5.3]{Wu},
without a constraint on the cardinality of $X$.

\subsection{Diffeologies determined by smooth linear functionals}

\begin{de}\label{de:DV}
 The diffeology on a diffeological vector space $V$ is \dfn{determined by its smooth linear functionals}
 if $p : U \to V$ is a plot if and only if $l \circ p$ is smooth for every $l \in L^{\infty}(V, \R)$.
 Write $\cDV$ for the collection of all such diffeological vector spaces.
\end{de}

 Note that any vector space with the indiscrete diffeology is in $\cDV$.
 It follows that being in $\cDV$ does not imply any of the other conditions
 we have studied.

 Also note that every diffeological vector space $V$ in $\cDV$ is \dfn{Fr\"olicher}:
 $p : U \to V$ is a plot if and only if $f \circ p$ is smooth for every $f \in C^{\infty}(V, \R)$.
 We do not know if the converse holds.

 We will see in Proposition~\ref{prop:DV=>equiv} that for diffeological vector spaces in $\cDV$,
 the converse of Theorem~\ref{thm:slfsp=>fdf} holds.
 For this, we need the following results.

\begin{lem}\
 \begin{enumerate}
  \item If $V$ is in $\cDV$ and $W$ is a subspace of $V$, then $W$ is in $\cDV$.
  \item Let $\{V_i\}$ be a set of diffeological vector spaces. Then each $V_i$ is in $\cDV$ 
   if and only if $\prod V_i$ is in $\cDV$.
 \end{enumerate}
\end{lem}

Since the category $\DVect$ is an additive category, (2) also implies that $\cDV$ is closed under taking finite direct sums.

\begin{proof}
 This is straightforward.
\end{proof}

\begin{prop}\label{prop:dv=>indiscrete+sv}
 Let $V$ be a diffeological vector space.
 Then $V$ is in $\cDV$ if and only if $V$ can be written as a direct sum
 $V \cong W_0 \oplus W_1$ of diffeological vector spaces,
 where $W_0$ is indiscrete and $W_1$ is in $\cSV \cap \cDV$.
\end{prop}

\begin{proof}
 Given $V$ in $\cDV$, let $W_0$ be $\cap_{l \in L^\infty(V,\R)} \ker(l)$ with the sub-diffeology.
 Since $L^\infty(V,\R)$ determines the diffeology on $V$, $W_0$ is indiscrete.
 Let $W_1$ be the quotient $V/W_0$, with the quotient diffeology.
 By the next lemma, we have $V \cong W_0 \oplus W_1$ as diffeological vector spaces.
 If $v + W_0$ is a non-zero element of $W_1$, then $v \not\in W_0$,
 so there is a smooth linear functional $l : V \to \R$ such that $l(v) \neq 0$.
 This $l$ factors through $W_1$, so it follows that $W_1$ is in $\cSV$.
 By the previous lemma, we know that $W_1 \in \cDV$, and hence $W_1 \in \cSV \cap \cDV$.

 The converse follows from the previous lemma and the comment after Definition~\ref{de:DV}.
\end{proof}

\begin{de}
Following \cite[Definition~3.15]{Wu}, a diagram
 \[
   0 \lra W_0 \llra{i} V \llra{p} W_1 \lra 0
 \]
of diffeological vector spaces is a \dfn{short exact sequence of diffeological vector spaces}
if it is a short exact sequence of vector spaces, $i$ is an induction, and $p$ is a subduction.
\end{de}

\begin{lem}
 Let
 \[
   0 \lra W_0 \llra{i} V \llra{p} W_1 \lra 0
 \]
 be a short exact sequence of diffeological vector spaces.
 If $W_0$ is indiscrete, then the sequence splits smoothly, so that
 $V \cong W_0 \oplus W_1$ as diffeological vector spaces.
\end{lem}

\begin{proof}
 Let $q : V \to W_0$ be any linear function such that $q \circ i = 1_{W_0}$.
 Since $W_0$ is indiscrete, $q$ is smooth.
 Let $k : V \to V$ be the smooth linear map sending $v$ to $v - i(q(v))$.
 Then $k \circ i = 0$, so $k$ factors as $j \circ p$, where $j : W_1 \to V$ is smooth and linear.
 The smooth bijection $V \to W_0 \oplus W_1$ sending $v$ to $(q(v), p(v))$ has a smooth inverse
 sending $(w_0, w_1)$ to $i(w_0) + j(w_1)$, so the claim follows.
\end{proof}

 It follows that many properties of a diffeological vector space are equivalent in this setting:

\begin{prop}\label{prop:DV=>equiv}
 Let $V$ be in $\cDV$.
 Then the following are equivalent:
 \begin{enumerate}
 \item $V$ is in $\cSV$.
 \item $V$ is in $\cSD$.
 \item $V$ is in $\cF$.
 \item $D(V)$ is Hausdorff.
 \item $V$ has no non-zero indiscrete subspace.
 \end{enumerate}
 Moreover, being in $\cDV$ and satisfying one of these conditions
 is equivalent to being a subspace of a product of copies of $\R$.
\end{prop}

\begin{proof}
 Without any assumption on $V$, we have (1) $\implies$ (2) $\implies$ (3) and (2) $\implies$ (4)
 using Theorem~\ref{thm:SD=>FDF} and Proposition~\ref{prop:SD=>H}.
 It is easy to see that (3) $\implies$ (5) and (4) $\implies$ (5).
 By Proposition~\ref{prop:dv=>indiscrete+sv}, (5) $\implies$ (1) when $V$ is in $\cDV$,
 and so we have shown that the five conditions are equivalent for $V \in \cDV$.

 For the last claim, a product of copies of $\R$ is in both $\cSV$ and $\cDV$,
 and both are closed under taking subspaces.
 Conversely, if $V$ is in $\cSV \cap \cDV$, it is easy to check that 
 \[
  V \to \prod_{L^\infty(V,\R)} \R
 \]
 defined by $v \mapsto (f(v))_{f \in L^\infty(V,\R)}$ is a linear induction, and hence $V$ is a 
 subspace of a product of copies of $\R$.
\end{proof}

\begin{rem}\
 \begin{enumerate}
  \item It is not true that every diffeological vector space is in $\cDV$.
        For example, when $\R$ is equipped with the continuous diffeology (see Example~\ref{ex:H=/=>SD}),
        all smooth linear functionals are zero, but the diffeology is not indiscrete.
  \item Other properties we have studied cannot be added to Proposition~\ref{prop:DV=>equiv}.
        For example, we saw in Remark~\ref{rem:SV>P}(1) that $\prod_{\omega} \R$ is in $\cSV$ but is
        not fine or projective.
        And it is easy to see that $\prod_{\omega} \R$ is in $\cDV$.
 \end{enumerate}
\end{rem}

 It is not hard to show that every fine diffeological vector space is in $\cDV$.
 As a final example, we will show that not every projective diffeological 
 vector space is in $\cDV$, and therefore that none of our other conditions
 on a diffeological vector space $V$ implies that $V$ is in $\cDV$.

 We will again use the diffeological vector space $F(\R)$, which is
 projective by~\cite[Corollary~6.4]{Wu}.
 We now show that it is not in $\cDV$.

\begin{prop}\label{prop:d-vs-slf}
 The free diffeological vector space $F(\R)$ generated by $\R$ is not in $\cDV$.
\end{prop}

\begin{proof}
 Fix a non-zero smooth function $\phi:\R \to \R$ such that $\supp(\phi) \subset (0,1)$ and 
 $|\phi(x)| \leq 1$ for all $x \in \R$.
 For each $n \in \Z^+$, define $\phi_n:\R \to \R$ by 
 \[
   \phi_n(x)=\phi\left(\frac{x-\frac{1}{n+1}}{\frac{1}{n}-\frac{1}{n+1}}\right).
 \]
 Finally, define $g:\R \ra F(\R)$ by 
 \[
 g(t)=\begin{cases} 2^{-n} \, \phi_n(t) \, \sum_{i=1}^n [\frac{1}{i}], \;
 & \textrm{if $\frac{1}{n+1} \leq t < \frac{1}{n}$, for $n > 0$} \\ 
                        0,                                          & \textrm{else.} \end{cases}
 \]
 Then $g$ is not a plot of $F(\R)$, since locally around $0 \in \R$, 
 $g$ cannot be written as a finite sum of $f_i(x) [h_i(x)]$,
 where $f_i$ and $h_i$ are smooth functions with codomain $\R$.
 But for each $l \in L^\infty(F(\R),\R)$, 
 \[
 l \circ g(t) = \begin{cases} 2^{-n} \, \phi_n(t) \, \sum_{i=1}^n l([\frac{1}{i}]), \;
                                                               & \textrm{if $\frac{1}{n+1} \leq t <      \frac{1}{n}$} \\ 
                                 0,                            & \textrm{else.} \end{cases}
 \]
 This is smooth, since the set $\{ l([\frac{1}{i}]) \}$ is bounded, using the smoothness of $l$.
\end{proof}

As an easy corollary, we have:

\begin{cor}
 $F(\R)$ is not a subspace of a product of copies of $\R$.
\end{cor}

\section{Some applications}\label{se:applications}

Recall that a diagram
\[
 \xymatrix{0 \ar[r] & V_1 \ar[r]^f & V_2 \ar[r]^g & V_3 \ar[r] & 0} 
\]
is a short exact sequence of diffeological vector spaces if it is a
short exact sequence of vector spaces such that $f$ is an induction and
$g$ is a subduction.
We say that the sequence \dfn{splits smoothly} if
there exists a smooth linear map $r:V_2 \to V_1$ such that $r \circ f = 1_{V_1}$, 
or equivalently, if there exists a smooth linear map $s:V_3 \to V_2$ such that 
$g \circ s = 1_{V_3}$.
In either case, $V_2$ is smoothly isomorphic to $V_1 \times V_3$.
(See~\cite[Theorem~3.16]{Wu}.)

Not every short exact sequence of diffeological vector spaces splits smoothly.
For example,
if we write $K$ for the subspace of $C^\infty(\R,\R)$ consisting of the smooth functions which 
are flat at $0$, then $K$ is not a smooth direct summand of $C^\infty(\R,\R)$~\cite[Example~4.3]{Wu}.

As a first application of the theory established so far, we can construct additional short exact 
sequences of diffeological vector spaces which do not split smoothly:

\begin{ex}
 Let $M$ be a manifold of positive dimension, and let $A$ be a finite subset of $M$. 
 Write $V$ for the subspace of $F(M)$ spanned by the subset $M \setminus A$ of $M$.
 We claim that $V$ is not a smooth direct summand of $F(M)$.

 To see this, write $W$ for the quotient diffeological vector space $F(M)/V$. 
 Then, as a vector space, $W = \oplus_{a \in A} \R$. So we have a short exact sequence 
 $0 \to V \to F(M) \to W \to 0$ in $\DVect$.
 Suppose this sequence splits smoothly.
 By Example~\ref{ex:projective}, $F(M)$ is projective, and therefore $W$ is as well.
 By Proposition~\ref{prop:PV=>SV} and Theorem~\ref{thm:slfsp=>fdf}, $W$ is in $\cF$.
 Since $W$ is finite-dimensional, it is fine.
 But the smooth map $M \to F(M) \to W = \oplus_{a \in A} \R$ 
 sends each $a \in A$ to a basis vector and other points in $M$ to $0$, so it is not
 a smooth map in the usual sense.  This contradicts the fact that $W$ is fine.
\end{ex}

As a second application, we prove:

\begin{thm}
 Let $V$ be in $\cSV$.
 Then every finite-dimensional subspace of $V$ is a smooth direct summand.
\end{thm}

\begin{proof}
 Let $W$ be a finite-dimensional subspace of $V \in \cSV$. By Theorem~\ref{thm:slfsp=>fdf}, 
 we know that $W$ has the fine diffeology. Moreover, since $V$ is in $\cSV$, there is a smooth 
 linear injective map $V \to \prod_{i \in I} \R$ for some index set $I$. 
 Since $\prod_{i \in I} \R$ is in $\cSV$, again by Theorem~\ref{thm:slfsp=>fdf}, we know 
 that the composite $W \hookrightarrow V \to \prod_{i \in I} \R$ is an induction, although 
 the second map might not be an induction. So, we are left to prove this statement for the 
 case $V=\prod_{i \in I} \R$. 

 Write $\dim(W)=m$. By Gaussian elimination, there exist distinct $i_1,\ldots,i_m \in I$ 
 such that the composite $W \hookrightarrow V=\prod_{i \in I} \R \to \R^m$ is an isomorphism of 
 vector spaces, where the second map is the projection onto the $i_1,\ldots,i_m$
 coordinates, and hence smooth. Since both $W$ and $\R^m$ have the fine diffeology, this
 isomorphism is a diffeomorphism, and by composing with its inverse we obtain
 a smooth linear map $r:V \to W$ such that the composite 
 \[
  \xymatrix{W\ \ar@{^{(}->}[r] & V \ar[r]^r & W}
 \]
 is $1_W$. Therefore, $W$ is a smooth direct summand of $V$.
\end{proof}

% Hint that this is a good choice:
\pagebreak[1]
\section{Proof of Theorem~\ref*{thm:SD=>FDF}}\label{se:SD=>FDF}

\theoremstyle{plain}
\newtheorem*{repeatthm}{Theorem~\ref*{thm:SD=>FDF}}
\begin{repeatthm}
 Every diffeological vector space in $\cSD$ is in $\cF$.
\end{repeatthm}

\begin{proof}
If a diffeological vector space is in $\cSD$, then so are all of its subspaces.
So it suffices to show that every finite-dimensional diffeological vector
space in $\cSD$ is fine.

Write $V$ for $\R^n$ with the structure of a diffeological vector space
which is not fine.
We will use the word ``smooth'' (resp.\ ``continuous'') to describe functions $\R \to V$ and
$V \to \R$ which are smooth (resp.\ continuous) with respect to the usual diffeology (resp.\ topology) on $\R^n$.
We use the word ``plot'' to describe functions $\R \to V$ which
are in the diffeology on $V$, and write $f \in C^{\infty}(V, \R)$
to describe functions which are smooth with respect to this diffeology.

By Proposition~\ref{prop:fineness-on-curves}, there is
a plot $p : \R \to V$ which is not smooth.
Since plots are closed under translation
in the domain and codomain, we can assume without loss of generality
that $p(0)=0$ and $p$ is not smooth at $0 \in \R$.
We will show that this implies that $V$ is not in $\cSD$.

\smallskip
\noindent
\textbf{Case 1:} Suppose that $p$ is continuous at $0$.  
Consider $A:=\{\nabla f(x) \mid f \in C^\infty(V,\R), x \in V\}$.
Then $A$ is a subset of $\R^n$. 

We claim that $A$ is a proper subset of $\R^n$.
If $A$ is not proper, then there exist 
$(f_1,a_1),\ldots,(f_n,a_n) \in C^\infty(V,\R) \times V$ such that
$\nabla f_1(a_1),\ldots,\nabla f_n(a_n)$ are linearly independent.
Then $g_i:V \to \R$ defined by $g_i(x)=f_i(x+a_i)$
is in $C^\infty(V,\R)$, $G:=(g_1,\ldots,g_n):V \to \R^n$ is smooth, 
and the Jacobian $JG(0)$ is invertible.
Therefore, by the inverse function theorem, $G$ is a local diffeomorphism near $0 \in V$.
Since $p(0)=0$, $p$ is continuous at $0 \in \R$, and $G \circ p$ is smooth,
it follows that $p$ is smooth at $0$, contradicting our assumption on $p$.
So $A$ is a proper subset.

By the same method of translation, one sees that $A$ is a subspace of $\R^n$.
Hence, there exists $0 \neq v \in \R^n$ such that $v \perp A$, which implies that $f(x+tv)=f(x)$ 
for every $f \in C^\infty(V,\R)$, $x \in V$ and $t \in \R$, i.e., 
$V$ is not in $\cSD$.

\smallskip
\noindent
\textbf{Case 2:} Suppose that $p$ is not continuous at $0$. 

\smallskip
\noindent
\textbf{Case 2a:}  Suppose there exist $k \in \N$ and $\eps > 0$
such that $t^k p(t)$ is bounded on $[-\eps,\eps]$.
Let $k$ be the smallest such exponent and
write $q(t) := t^k p(t)$, which is also a plot.
We claim that $q$ is not smooth at $0$.
If $k = 0$, then $q = p$, which is assumed to not be smooth at $0$.
If $k > 0$ and $q'(0)$ exists, then $q(t)/t \to q'(0)$ as $t \to 0$,
which implies that $t^{k-1} p(t)$ is also bounded on $[-\eps,\eps]$,
contradicting the minimality of $k$.
So $q$ is not smooth at $0$.

If $q$ is continuous at $0$, then by Case 1, we are done.

So assume that $q$ is not continuous at $0$.
Then, since $q$ is bounded on $[-\eps,\eps]$, there exists a sequence
$t_i$ converging to $0$ such that $q(t_i)$ converges to a non-zero $v \in V$.
If $f$ is in $C^{\infty}(V, \R)$, then $f \circ q$ is smooth, 
so $f(0) = f(q(0)) = f(q(\lim t_i)) = \lim f(q(t_i)) = f(\lim\, q(t_i)) = f(v)$.
Therefore, the functions in $C^{\infty}(V, \R)$ do not separate points.

\smallskip
\noindent
\textbf{Case 2b:}  Suppose that Case 2a does not apply.
Then for each $k \in \N$, $\eps > 0$ and $M > 0$, 
there exists $t \in [-\eps,\eps]$ such that $\|t^k \, p(t)\| > M$.
(Note that $t \neq 0$, since $p(0) = 0$.)
Using this for $k = 0$, choose $t_1 \in [-1, 1]$ such that $\|p(t_1)\| > 1$.
Then, for each integer $k > 0$, choose $t_k$ with $|t_k| \leq |t_{k-1}|/2$
such that $\|t_k^k \, p(t_k)\| > k$.
If $m \leq k$, then $t_k$ also satisfies $\|t_k^m \, p(t_k)\| > k \geq m$,
since $|t_k| \leq 1$.
Therefore, we can restrict to a subsequence of the $t_k$ all having the same sign.
To fix notation, assume that each $t_k$ is positive.
Then, for $m \leq k$,
\[
  \frac{\,\,\frac{1}{\|p(t_k)\|}\,\,}{t_k^m} < \frac{1}{k}
\]
and so, for each $m$, the left-hand-side goes to $0$ as $k \to \infty$.
By Lemma~\ref{le:escl} below, there is a smooth curve
$c : \R \to \R$ such that $c(t_k) = 1/\|p(t_k)\|$.
It follows that $q(t) := c(t) p(t)$ is a plot such that $q(0) = 0$
and $\|q(t_k)\| = 1$ for each $k$.
Therefore, there is a subsequence converging to a non-zero $v \in V$,
and the argument at the end of Case 2a shows that 
$C^{\infty}(V, \R)$ does not separate points.
\end{proof}

\begin{lem}[Extended special curve lemma]\label{le:escl}
Let $\{ x_k \}$ and $\{ t_k \}$ be sequences in $\R$ such that 
$0 < t_k < t_{k-1}/2$ for each $k$ and $x_k/t_k^m \to 0$ as $k \to \infty$ for each $m \in \Z^+$.
Then there is a smooth function $c : \R \to \R$ such that $c(t_k) = x_k$ for each $k$
and $c(t) = 0$ for $t < 0$.
\end{lem}

The proof closely follows~\cite[page~18]{KM}, and can easily be
generalized further.

\begin{proof}
Let $\phi : \R \to \R$ be a smooth function such that $\phi(t) = 0$ for $t \leq 0$
and $\phi(t) = 1$ for $t \geq 1$.  Define $c : \R \to \R$ by
\[
  c(t) = \begin{cases}
           0, & \text{for $t \leq 0$} \\
           x_{k+1} + \phi\left(\frac{t - t_{k+1}}{t_k - t_{k+1}}\right)(x_k - x_{k+1}), \; & \text{for $t_{k+1} \leq t \leq t_k$} \\
           x_1, & \text{for $t_1 \leq t$.}
         \end{cases}
\]
$c$ is smooth away from $0$ and for $t_{k+1} \leq t \leq t_k$ we have
\[
  c^{(r)}(t) = \phi^{(r)}\left(\frac{t - t_{k+1}}{t_k - t_{k+1}}\right) \frac{1}{(t_k - t_{k+1})^r}  (x_k - x_{k+1}) .
\]
Since $t_k - t_{k+1} > t_k/2 > t_{k+1}$, the right-hand-side goes to zero as $t \to 0$.
Similarly, $c^{(r)}(t)/t \to 0$, which shows that each $c^{(r+1)}(0)$ exists and is $0$.
So $c$ is smooth.
\end{proof}

We are indebted to Chengjie Yu for the argument used in Case 1 of Theorem~\ref{thm:SD=>FDF}.
After we completed Case 2, Yongjie Shi and Chengjie Yu proved this case in
more generality in~\cite{SY}.

\end{document}